\documentclass[11pt]{amsart}
\usepackage{hyperref}

\renewcommand{\d}{\mathrm{d}}

\newcommand{\cI}{{\mathcal I}}
\newcommand{\cV}{{\mathcal V}}
\newcommand{\bbR}{{\mathbb R}}

\newcommand{\bbP}{{\mathbb P}}

\newcommand{\GL}{\operatorname{GL}}
\newcommand{\Gr}{\operatorname{Gr}}
\newcommand{\Aut}{\operatorname{Aut}}

\newcommand{\w}{{\mathchoice{\,{\scriptstyle\wedge}\,}{{\scriptstyle\wedge}}
      {{\scriptscriptstyle\wedge}}{{\scriptscriptstyle\wedge}}}}
\newcommand{\lhk}{\mathbin{\hbox{\vrule height1.4pt width4pt depth-1pt 
             \vrule height4pt width0.4pt depth-1pt}}}

\newcommand{\be}{\begin{equation}}
\newcommand{\ee}{\end{equation}}
\newcommand{\bpm}{\begin{pmatrix}}
\newcommand{\epm}{\end{pmatrix}}

\numberwithin{equation}{section}

\newtheorem{theorem}{Theorem}
\newtheorem{lemma}{Lemma}

\theoremstyle{remark}

\newtheorem{remark}{Remark}

\begin{document}

\author[R. Bryant]{Robert L. Bryant}
\address{Duke University Mathematics Department\\
         P.O. Box 90320\\
         Durham, NC 27708-0320}
\email{\href{mailto:bryant@math.duke.edu}{bryant@math.duke.edu}}
\urladdr{\href{http://www.math.duke.edu/~bryant}%
         {http://www.math.duke.edu/\lower3pt\hbox{\symbol{'176}}bryant}}

\title[Hessianizability of surface metrics]
      {Hessianizability of surface metrics}
      
\dedicatory{Dedicated to the memory of Joseph A.~Wolf, whose generous advice, inspiring works, and seemingly inexhaustible knowledge of geometry has been a boon to me and so many others.}

\date{January 30, 2024}

\begin{abstract}
A symmetric quadratic form $g$ on a surface~$M$ is said
to be \emph{locally Hessianizable} if each $p\in M$ has an open
neighborhood~$U$ on which there exists a local coordinate chart
$(x^1,x^2):U\to\mathbb{R}^2$ and a function $f:U\to\mathbb{R}$
such that, on $U$, we have
$$
g = \frac{\partial^2 f}{\partial x^i\partial x^j}\,\d x^i\circ\d x^j.
$$
In this article, I show that, when $g$ is nondegenerate and smooth,
it is always smoothly locally Hessianizable.
\end{abstract}



\thanks{
Thanks to Duke University for its support via a research grant, 
to the NSF for its support via DMS-0103884 and DMS-1359583
(originally DMS-1105868), 
and to Columbia University for its support via 
a Samuel Eilenberg Visiting Professorship during Spring 2004, 
during which this work was begun (and mostly completed).
}

\maketitle

\setcounter{tocdepth}{2}
\tableofcontents

\section{Introduction}\label{sec: intro} 

A symmetric quadratic differential form $g$ on an $n$-manifold $M^n$ is said to be of \emph{Hessian type} if there exist $(n{+}1)$ functions $x^1,\ldots,x^n, f$ on~$M$ such that $\d x^1\wedge\cdots\wedge\d x^n\not=0$ and such that%
\footnote{Note the use of the Einstein summation convention, which will be employed throughout this article. Also, note the standard use of $\alpha\circ\beta$ 
for $\tfrac12(\alpha\otimes\beta+\beta\otimes\alpha)$, i.e., the symmetrized tensor product.}
$$
g = \frac{\partial^2f}{\partial x^i\partial x^j}\, \d x^i\circ \d x^j\,.
$$
Note that the independence of the differentials $\d x^i$ is needed in order to define the `partial derivatives' in the formula.  One says that $g$ is \emph{locally of Hessian type} if each point of $M$ has an open neighborhood $U\subset M$ on which there exists a coordinate chart $x:U\to\mathbb{R}^n$ and a function $f\in C^\infty(U)$ such that the above formula holds on $U$.  The notion of Hessian type has been considered in many places in the literature (cf.~\cite{Duistermaat2001,MR2293045,Totaro2004} and the papers cited therein) and is connected with affine and K\"ahler geometry and many other areas.

Since symmetric quadratic differential forms $g$ in dimension $n$ depend on $\tfrac12n(n{+}1)$ functions of $n$ variables and the data of a Hessian representation depends only on $(n{+}1)$ functions of $n$ variables, it is clear that, when $n>2$, not every $g$ is locally of Hessian type.  As $n$ increases, the condition of being locally of Hessian type becomes more and more restrictive.  When $g$ is nondegenerate, so that $(M^2,g)$ is a pseudo-Riemannian manifold, being of Hessian type implies algebraic conditions on the Riemann curvature tensor once $n$ is sufficiently large~\cite{Totaro2004,AmariArmstrong2014}.

Meanwhile, when $n=2$, finding Hessianizing data~$(x^1,x^2,f)$ for a given quadratic form $g$ is a determined PDE problem, being three equations for three unknowns, so one might expect formal solvability. However, it is not necessarily true that the PDE system involved is nondegenerate, so even formal solvability is not guaranteed.

For example, the smooth (in fact, real-analytic) quadratic form 
$g = c_{ijk} y^i\,\d y^j\circ\d y^k$, where $c_{ijk}=c_{ikj}$ are constants, is not Hessianizable on a neighborhood of $y=0$ unless $c_{ijk}=c_{jik}$.

However, when $g\in S^2(T^*M)$ is real-analytic \emph{and nondegenerate} at~$p\in M^2$, it is straightforward to show, using the Cartan-K\"ahler Theorem, that $g$ is Hessianizable on an open neighborhood of~$p$.  See Theorem~\ref{thm: involutivity} below, which is the result  to which I alluded in~\cite{MO122308}.  (The reader might also want to compare the argument in~\cite{AmariArmstrong2014}, which was arrived at independently.)

It is natural to ask whether this Hessianizability result continues to hold in the smooth setting.  It turns out that the PDE involved is never elliptic, so one never gets existence and/or regularity by applying results from elliptic theory.  In fact, the complex characteristic variety of the PDE along regular initial conditions consists of $3$ points, so at least one of the characteristics has to be real.  

Now, it turns out that, when $g$ is nondegenerate, one can always find regular initial conditions for the PDE with $3$ real, distinct characteristics, i.e., initial conditions for which the PDE problem is hyperbolic.  Thus, one finds that one can invoke results of Yang~\cite{MR0897707} on involutive hyperbolic systems to prove smooth local Hessianizability of nondegenerate $g$ that are merely smooth.  See Theorem~\ref{thm: smooth}.

For details on the notation and definitions in the use of the Cartan-K\"ahler Theorem, which will be used to prove Theorem~\ref{thm: involutivity}, the reader is referred to~\cite{MR1992h:58007}.

\section{The structure equations}
Let $M^2$ be a surface, and let $g$ be a metric on~$M$,
i.e., a nondegenerate quadratic form on $M$.  
Thus, in a local coordinate chart $y = (y^i):U\to\bbR^2$, 
one has $g = g_{ij}(y)\,\d y^i\,\d y^j$, 
where $g_{11}g_{22}-g_{12}^2$ is nonvanishing.  
(One does not need to assume that $g$ be definite.)  

Let $\pi:F\to M$ denote the coframe bundle of~$M$, 
whose typical element~$u\in F$ is an isomorphism~$u:T_{\pi(u)}M\to\bbR^2$,
with its right action by~$A\in\GL(2,\bbR)=\Aut(\bbR^2)$ 
defined to be~$R_A(u) = u\cdot A = A^{-1}\circ u$.  
There is an $\bbR^2$-valued~$1$-form~$\omega:TF\to\bbR^2$ 
defined by~$\omega_u(v) = u\bigl(\pi'(v)\bigr)$ for~$v\in T_uF$.
Its components~$\omega = (\omega^i)$ 
are the canonical (`soldering') forms~$\omega^i$ 
(where $1\le i\le 2$). 
Let $G_{ij} = G_{ji}$ be the functions on~$F$ 
such that $\pi^*g = G_{ij}\,\omega^i{\circ}\omega^j$
and, as usual, let $G^{ij}=G^{ji}$ satisfy~$G^{ik}G_{kj} = \delta^i_j$.

Let $\omega^i_j$ be the Levi-Civita connection forms of~$g$ on~$F$.
These forms, together with the $\omega^i$, define a coframing of~$F$,
and they are uniquely specified by the requirements that they satisfy
\emph{compatibility} with $g$,
\begin{equation}\label{eq: Gcompat}
\d G_{ij} = G_{ik}\,\omega^k_j + G_{kj}\,\omega^k_i\,,
\end{equation}
and the \emph{first structure equation},
\begin{equation}\label{eq: firstSE}
\d\omega^i = -\omega^i_j\w\omega^j.
\end{equation}
Furthermore, they satisfy the \emph{second structure equation},
\begin{equation}\label{eq: secondSE}
\d\omega^i_j = -\omega^i_k\w\omega^k_j + G_{jl} K\,\omega^i\w\omega^l,
\end{equation}
where the function $K$ is the \emph{Gauss curvature} of~$g$ pulled up to $F$
via the projection $\pi:F\to M$.

It will be convenient to define $\omega_{ij}=G_{ik}\,\omega^k_j$,
so that the equations above become $\d G_{ij} = \omega_{ij}+\omega_{ji}$,
and to define $\omega_i = G_{ij}\,\omega^j$, 
so that $\pi^*g = \omega_i\circ\omega^i$.  
Note that one has $\omega_i\w\omega^i = 0$, since $G_{ij}= G_{ji}$.

\section{An exterior differential system for Hessianizability}
If~$U\subset M$ is an open subset on which there exist coordinates
$x = (x^i):U\to\bbR^2$ and a function $f$ such that, on~$U$, one has
$$
g = \frac{\partial^2 f}{\partial x^i\partial x^j}\,\d x^i\circ\d x^j
   = \d\left(\frac{\partial f}{\partial x^1}\right)\circ \d x^1 
   + \d\left(\frac{\partial f}{\partial x^2}\right)\circ \d x^2,
$$
then the coordinate coframing~$\xi = (\d x^i)$ defines a section $\xi:U\to F$
that satisfies $\xi^*(\omega^i) = \d x^i$ 
and $\xi^*(G_{ij}) = \frac{\partial^2 f}{\partial x^i\partial x^j}$.  
Consequently, $\xi^*(\omega_i) = \d\left(\frac{\partial f}{\partial x^i}\right)$.  
In particular, the section $\xi:U\to F$ satisfies
$$
\xi^*\bigl(\d\omega^i\bigr)  = \d\bigl(\d x^i)= 0
\quad\text{and}\quad
\xi^*(\d\omega_i)  
= \d\left(\d \left(\frac{\partial f}{\partial x^i}\right)\right)= 0
$$
while $\xi^*(\omega^1\w\omega^2) = \d x^1\w\d x^2\not=0$.

Conversely
\footnote{Here and throughout this article, when $S\subset M$ is a submanifold
and $\tau$ is a contravariant tensor field on $M$ (e.g., a differential form
or a metric on $M$), the notation $S^*(\tau)$ 
denotes the \emph{pullback} of $\tau$ to $S$, which the reader should take
care not to confuse with $\tau_{|S}$, i.e., the \emph{restriction} of $\tau$ to $S$.}
, if $S\subset F$ is a surface such that $S^*(\d\omega^i) = S^*(\d\omega_i) = 0$ 
while $S^*(\omega^1\w\omega^2)\not=0$, then
each point of~$S$ has an open neighborhood~$V\subset S$ 
that is the image of a local section $\xi:U\to F$ for some open set $U\subset M$.
When $U$ is simply connected, one has $\xi^*\omega^i = \d x^i$ for some
independent functions $x^i$ on $U$, while $\xi^*\omega_i = \d p_i$ for 
some functions $p_i$ on~$U$.  Since $\omega_i\w\omega^i = 0$, 
it follows that $\d(p_i\,\d x^i) = \d p_i\w\d x^i = 0$, 
so that there exists a function~$f$ on~$U$ such that $\d f =  p_i\,\d x^i$,
i.e., $p_i = \frac{\partial f}{\partial x^i}$.
Thus,
$$
g = \xi^*(\omega_i\circ\omega^i) = \d p_i\circ\d x^i 
  =\frac{\partial^2 f}{\partial x^i\partial x^j}\,\d x^i\circ\d x^j,
$$
as desired.

Letting $\cI$ be the differential ideal on~$F$
generated by the four $2$-forms~$\d\omega^i$ and $\d\omega_i$ for $i=1,2$,
it then follows that the problem of locally writing $g$ in Hessian form 
is equivalent to finding the integral surfaces of~$\cI$
on which the $2$-form $\omega^1\w\omega^2$ is nonvanishing. 

\subsection{A prolongation}
The ideal with independence condition
$\bigl(\cI,\omega^1\w\omega^2\bigr)$ is not quite the right ideal
to consider, for it does not contain all of the forms on~$F$ 
that vanish on Hessianizing sections.  To see this, note that, 
by definition, $G_{ij}\d\omega^j = - \omega_{ij}\w\omega^j$
while, since $\d G_{ij} = \omega_{ij}+\omega_{ji}$, one has that
$$
\d\omega_i = \d G_{ij}\w\omega^j + G_{ij}\d\omega^j = \omega_{ji}\w\omega^j.
$$
Thus,~$\cI$ is generated by the four $2$-forms~$\omega_{ij}\w\omega^j$
and~$\omega_{ji}\w\omega^j$  for $i=1,2$.  In particular, $\cI$ contains the
pair of $2$-forms~$(\omega_{12}-\omega_{21})\w\omega^j$ for $j=1,2$, 
and, consequently, any integral surface of~$\cI$ 
on which $\omega^1\w\omega^2\not=0$
must also be an integral surface of the $1$-form 
$\theta = \tfrac12(\omega_{12}-\omega_{21})$. 
Thus, one can add~$\theta$ and its exterior derivative $\d\theta$ 
to~$\cI$ to form a larger ideal~$\cI_+$ with the property
that the integral surfaces of $\bigl(\cI_+,\omega^1\w\omega^2)$
are the same as the integral surfaces of $\bigl(\cI,\omega^1\w\omega^2)$.
Henceforth, I will work with this larger ideal, which is a partial
prolongation of~$\bigl(\cI,\omega^1\w\omega^2)$.

Write $\sigma_{ij} = \tfrac12(\omega_{ij}+\omega_{ji})$ 
($= \tfrac12\d G_{ij}$), so that $\omega_{ij}\equiv\sigma_{ij}\mod\theta$
and 
$$
-G_{ij}\d\omega^j\equiv\d\omega_i\equiv \sigma_{ij}\w\omega^j\mod\theta.
$$
Note that the forms $\omega^1$,
$\omega^2$, $\theta$, $\sigma_{11}$, $\sigma_{12}$, and $\sigma_{22}$ 
are everywhere linearly independent and hence define a coframing of~$F$.
Using the second structure equations, one finds
$$
\d\theta \equiv G^{kl}\sigma_{k1}\w\sigma_{l2} + K\,\omega_1\w\omega_2
\mod \theta.
$$ 

Thus, the differential ideal $\cI_+$ is generated algebraically by $\theta$
and three $2$-forms:
\begin{equation}\label{eq: I+ 2-forms}
\begin{aligned}
\Omega_1 &= \sigma_{1j}\w\omega^j\,,\\
\Omega_2 &= \sigma_{2j}\w\omega^j\,,\\
\Theta &= G^{kl}\sigma_{k1}\w\sigma_{l2} + K\,\omega_1\w\omega_2\,.
\end{aligned}
\end{equation} 

It is easy to see that the right action of $\mathrm{GL}(2,\mathbb{R})$
on $F$ preserves the ideal 
with independence condition~$\bigl(\cI_+,\omega^1\w\omega^2\bigr)$.
In particular, this right action prolongs to an action of~$\mathrm{GL}(2,\mathbb{R})$
on the space $\cV_2\bigl(\cI_+,\omega^1\w\omega^2\bigr)\subset \Gr_2(TF)$, 
which consists of the $2$-dimensional integral elements 
of $\cI_+$ on which $\omega^1\w\omega^2$ is nonvanishing. 

\section{Involutivity}
The pieces are now in place to prove the main result. 

\begin{theorem}\label{thm: involutivity} 
The differential ideal with independence condition~$\bigl(\cI_+,\omega^1\w\omega^2\bigr)$
is involutive, with Cartan characters~$s_0=1$, $s_1=3$, and $s_2 = 0$.  

If $g$ is real analytic, 
then the integral surfaces of~$\bigl(\cI_+,\omega^1\w\omega^2\bigr)$
depend on three arbitrary functions of one variable.
In particular, each point of~$M$ has an open neighborhood on which $g$
can be written in Hessian form.
\end{theorem}

\begin{proof}
Because $\cI_+$ is generated in degree~$1$ 
by the nonvanishing $1$-form~$\theta$, all of the $0$-dimensional
integral elements are both ordinary and regular, and all of its 
$1$-dimensional integral elements are ordinary.

Thus, to prove involutivity, it suffices to show that, at each~$u\in F$, 
there exists a regular $1$-dimensional integral element~$E_1\subset T_uF$
for which~$\omega^1\w\omega^2$ is nonvanishing on the polar space~$H(E_1)$.  

To this end, let $v\in T_uF$ be the basis 
of a $1$-dimensional integral element, 
i.e., $\theta(v) = 0$ while~$v$ is nonzero.  
Then the polar space of~$E_1 = \bbR v$ is
$$
H(E_1) 
= \{\ w\in T_uF\ \vrule\ 
   \theta(w)=\Omega_1(v,w)=\Omega_2(v,w)=\Theta(v,w) = 0\ \},
$$
and hence has dimension at least~$2$.
Set $a^i = \omega^i(v)$ and $s_{ij} = \sigma_{ij}(v)$.  Then
\begin{equation}
\begin{aligned}
v\lhk\Omega_1 
&= -a^1\,\sigma_{11}-a^2\,\sigma_{12}+s_{11}\,\omega^1+s_{12}\,\omega^2,\\
v\lhk\Omega_2 
&= -a^1\,\sigma_{12}-a^2\,\sigma_{22}+s_{12}\,\omega^1+s_{22}\,\omega^2,\\
v\lhk\Theta\phantom{_2} 
&=  G^{kl}(s_{k1}\sigma_{l2}-s_{l2}\sigma_{k1}) 
+ K(G_{11}G_{22}-{G_{12}}^2)\,(a^1\,\omega^2-a^2\,\omega^1).\\
\end{aligned}
\end{equation}
Thus, $H(E_1)$ has dimension~$2$ 
and $\omega^1\w\omega^2$ is nonvanishing on~$H(E_1)$ exactly when
\begin{equation}\label{eq: inequality1}
(a^1)^2\,s^2_1 - a^1a^2\,(s^1_1-s^2_2) - (a^2)^2\,s^1_2\not=0,
\end{equation}
where, for brevity, I have written $s^i_j = G^{ik}(u)s_{kj}$.  
In particular, except when $s_{ij} = \lambda G_{ij}(u)$ for some $\lambda$,
it will always be possible to choose $a^1$ and $a^2$ 
so that the inequality~\eqref{eq: inequality1} holds. 
For a $v\in T_uF$ so chosen, 
the integral element $E_1=\bbR v$ is regular 
and, hence, its $2$-dimensional polar space~$H(E_1)=E_2$ 
is an ordinary integral element of~$\bigl(\cI_+,\omega^1\w\omega^2\bigr)$.
Moreover, the characters of the ordinary flag $(0)\subset E_1\subset E_2$ 
are visibly $s_0=1$, $s_1=3$, and $s_2 = 0$ 
and $\bigl(\cI_+,\omega^1\w\omega^2\bigr)$ is involutive.

Moreover, when $g$ is real analytic, 
the Cartan-K\"ahler Theorem applies, 
yielding that any analytic curve~$\Gamma\subset F$ 
that satisfies $\Gamma^*\theta=0$ and $\Gamma^*C\not=0$,
where~$C$ is the characteristic cubic form%
\footnote{For use below,
I note that $R_A^*(C) = \det(A)\,C$ for $A\in\GL(2,\bbR)$.}
\begin{equation}\label{eq: Ccharcubic}
C = (\omega^1)^2\circ (G^{2j}\sigma_{j1}) 
    - \omega^1\circ\omega^2\circ (G^{1j}\sigma_{j1}-G^{2j}\sigma_{j2})
    - (\omega^2)^2\circ (G^{1j}\sigma_{j2}), 
\end{equation}
lies in a unique (ordinary) analytic integral surface 
of~$\bigl(\cI_+,\omega^1\w\omega^2\bigr)$.  

When $g$ (and hence~$\cI_+$) is real analytic, 
it is easy to construct such analytic curves~$\Gamma\subset F$:
Let $A\subset M$ be any connected, noncompact, analytic,
embedded curve, on which there is a parameter~$t:A\to(a,b)\subset \bbR$ 
and let $F_A = \pi^{-1}(A)\subset F$.
Then $F_A$ is a $5$-dimensional embedded, analytic submanifold
on which $\omega^1\w\omega^2$ vanishes, although $\omega^1$ 
and~$\omega^2$ do not simultaneously vanish.  
Thus, there exist two functions~$a^i:F_A\to\bbR$ for $i=1,2$ 
that do not simultaneously vanish such that $\omega^i=a^i\,\pi^*(\d t)$
on~$F_A$.  Then let $\Gamma\subset F_A$ be any real analytic
curve in~$F_A$ that projects diffeomorphically onto~$A$
and satisfies~$\Gamma^*\theta=0$ while $\Gamma^*\psi\not=0$
where~$\psi$ is the $1$-form on~$F_A$ defined by
$$
\psi = (a^1)^2\,G^{2k}\sigma_{k1} 
- a^1a^2\,(G^{1k}\sigma_{k1}-G^{2k}\sigma_{k2}) 
 - (a^2)^2\,G^{1k}\sigma_{k2}\,.
$$
Such a $\Gamma$ will then be a regular integral curve 
of~$\cI_+$ lying in a unique integral surface~$S\subset F$
of~$\bigl(\cI_+,\omega^1\w\omega^2\bigr)$.  
Moreover, an open neighborhood of~$\Gamma$ in~$S$ 
will project diffeomorphically onto an open neighborhood~$U$ 
of $A$ in~$M$, and the inverse map~$\xi:U\to S$ will be
a Hessianizing section for~$g$ on~$U$.

In particular, for any given point~$u\in F$,
an integral surface of~$\bigl(\cI_+,\omega^1\w\omega^2\bigr)$
exists that passes through~$u$.
  
Since the Cartan characters of the system are $s_0=1$, $s_1=3$,
and $s_2=0$, it follows that the ordinary integral surfaces
depend on three functions of one variable~\cite{MR1992h:58007}.

As explained above, this implies that the metric~$g$ 
is locally Hessianizable if it is real analytic.
\end{proof} 

\section{Hyperbolicity and smooth Hessianizability}
\label{ssec: hyperbolicity}

In order to prove Hessianizability in the smooth category, 
it will be necessary to take a closer look at the structure 
of the ideal with independence condition~$(\cI_+,\omega^1\w\omega^2)$.

\subsection{The characteristic variety}
Let~$\cV_2(\cI_+,\omega^1\w\omega^2)\subset\Gr_2(TF)$
denote the locus of $2$-dimensional integral elements 
of~$(\cI_+,\omega^1\w\omega^2)$.  A $2$-plane~$E\subset T_uF$
belongs to~$\cV_2(\cI_+,\omega^1\w\omega^2)$ if and only
if it is defined by $4$ Pfaffian equations of the form
\begin{equation}
\theta = \sigma_{ij}- p_{ijk}(E)\,\omega^k = 0,
\end{equation}
where the numbers $p_{ijk}(E)=p_{jik}(E)$ satisfy
the linear relations $p_{ijk}(E)=p_{ikj}(E)$
(which are equivalent to the vanishing of $\Omega_i$ on~$E$) 
and the quadratic relation
\begin{equation}\label{eq: E^*Theta=0}
\begin{aligned}
0 &= G^{kl}(u)\bigl(p_{k11}(E)p_{l22}(E){-}p_{k12}(E)p_{l12}(E)\bigr)\\
&\qquad\qquad + K(u)\bigl(G_{11}(u)G_{22}(u){-}G_{12}(u)^2\bigr)
\end{aligned}
\end{equation}
(which is equivalent to the vanishing of~$\Theta$ on~$E$).
The quadratic form in the $4$ independent quantities
$p_{111}(E)$, $p_{112}(E)$, $p_{122}(E)$, and $p_{222}(E)$
that makes up the first term of this relation 
is easily seen to be nondegenerate and indefinite.  
Hence, this exhibits $\cV_2(\cI_+,\omega^1\w\omega^2)$ 
as a hypersurface in $F\times\bbR^4$ whose fiber
over~$u\in F$ is a (nonempty) $3$-dimensional hyperquadric 
that is smooth away from the $2$-plane defined by $p_{ijk}(E)=K(u)=0$.

Meanwhile, the integral element~$E$ will be ordinary 
(i.e., it will contain a regular $1$-dimensional integral element) 
if and only if $E^*(C)\not=0$, 
since $\Xi_E\subset\bbP E^*$, i.e., the characteristic variety of~$E$, 
consists of the points represented by the linear factors of the cubic form 
\begin{equation}\label{eq: E^*C}
E^*(C)=\epsilon_{jk}G^{kl}(u)p_{lmn}(E)\,\,\omega^j\circ\omega^m\circ\omega^n.
\end{equation}
(Here, $\epsilon_{ij}=-\epsilon_{ji}$ and $\epsilon_{12}=1$.)
It follows that $E$ is ordinary 
unless~$p_{ijk}(E)=0$ for all $(i,j,k)$ 
which, by the above quadratic relation, can only happen if $K(u)=0$.

\begin{remark}[Non-ordinary integral surfaces]
An integral surface of the ideal $\bigl(\cI_+,\omega^1\w\omega^2\bigr)$
on which $C$ vanishes identically is a so-called `non-ordinary'
integral surface, i.e., it is an integral surface of $\bigl(\cI_+,\omega^1\w\omega^2\bigr)$ that contains no regular integral curves. 
Since $E^*(C)$ vanishes if and only if $p_{ijk}(E)=0$, 
it follows from \eqref{eq: E^*Theta=0} 
that non-ordinary integral surfaces 
can only exist over open sets in~$M^2$ on which the curvature $K$ vanishes,
and, in this case, they are the integral surfaces of the Frobenius
system generated by~$\theta$ and the $\sigma_{ij}$.  
Thus, when~$K\equiv0$, the coframe bundle~$F$ is foliated by
the $4$-parameter family of leaves of this system.

Conversely, if $g$ is a nondegenerate quadratic form on~$M$ 
with vanishing curvature, then each point~$p\in M$ has a neighborhood
on which there exist $p$-centered local coordinates~$(x^1,x^2)$
in which $g = c_{ij}\,\d x^i\,\d x^j$, 
where the $c_{ij}=c_{ji}$ are constants, 
so that $g$ is in Hessian form with $f = \tfrac12c_{ij}\,x^ix^j$.  
Such constant coefficient Hessianizations 
describe the non-ordinary integral surfaces.   
\end{remark}

\subsection{Hyperbolicity}

The (complex) characteristic variety of the ideal $\cI_+$ 
meets each cotangent space of each ordinary integral surface 
in the three linear factors of the cubic form~$C$ 
(as defined in~\eqref{eq: Ccharcubic}).  Since a nonvanishing cubic
in $2$ variables always has at least one real linear factor,
it follows that the real characteristic variety is never empty,
so the system is never elliptic.  
Therefore, the best behavior one could hope for is that the
characteristic variety at each point would consist of three
real distinct factors, which would be the hyperbolic case.

Recall that, for a cubic form~$c\in S^3(E^*)$, 
where $E$ is a $2$-dimensional vector space over $\mathbb{R}$, 
its discriminant $\Delta(c)$ is an element of $S^6\bigl(\Lambda^2(E^*)\bigr)$.  
When $(e^1,e^2)$ is a basis of $E^*$ and $c$ is expressed in the form
\begin{equation}\label{eq: cubic form c}
c = c_{0}\,(e^1)^3 + 3c_{1}\,(e^1)^2e^2+ 3c_{2}\,e^1(e^2)^2 + c_{3}\,(e^2)^3,
\end{equation}
then 
\begin{equation}\label{eq: discrim c}
\Delta(c)
= \bigl(6{c_{0}}{c_{1}}{c_{2}}{c_{3}}+3{c_{1}}^2{c_{2}}^2
-4{c_{0}}{c_{2}}^3-4{c_{1}}^3{c_{3}}-{c_{0}}^2{c_{3}}^2
\bigr)\,(e^1\w e^2)^{\otimes 6}.
\end{equation}
Then $\Delta(c)=0$ if and only if $c$ has a multiple linear factor.
Because the scalar field is $\mathbb{R}$, 
the inequality $\Delta(c)>0$ makes sense, and this inequality holds 
if and only if $c$ has three real, distinct, linear factors.

Let~$\cV^h_2(\cI_+,\omega^1\w\omega^2)\subset\cV_2(\cI_+,\omega^1\w\omega^2)$
denote the open subset that consists of integral elements~$E$
for which the characteristic cubic~$E^*(C)$ is nonzero 
and has three real, distinct, linear factors.  
By definition, the elements of $\cV^h_2(\cI_+,\omega^1\w\omega^2)$
are ordinary. I will refer to $\cV^h_2(\cI_+,\omega^1\w\omega^2)$ 
as the \emph{hyperbolic locus}. 

\begin{lemma}\label{lem: hyperbolic}
The basepoint projection $\cV^h_2(\cI_+,\omega^1\w\omega^2)\to M$ 
is a surjective submersion.
\end{lemma}

\begin{proof}
It suffices to prove that the basepoint projection 
$\cV^h_2(\cI_+,\omega^1\w\omega^2)\to M$ is surjective, 
since every element of $\cV^h_2(\cI_+,\omega^1\w\omega^2)$ is a Kähler-regular
integral element and the ideal $\cI_+$ is generated in positive degree.

Also, because the right action by $\mathrm{GL}(2,\mathbb{R})$ on $F$ is transitive
on the fibers of $F\to M$, it suffices to consider an element $u\in F$
at which all of $G_{11}$, $G_{12}$ and $G_{22}$ are nonzero, so I will assume this
hencforth.

Consider the locus $L_u\subset\cV_2(\cI_+,\omega^1\w\omega^2)$ in
the fiber over $u$ that is defined
by the two equations 
$$
G_{22}p_{122}-G_{12}p_{222}=0
\quad\text{and}\quad 
G_{12}p_{111}-G_{11}p_{112}=0.
$$
Because $G_{11}G_{22}-(G_{12})^2\not=0$, 
there will exist unique functions $s_1$ and $s_2$ on $L$ such that 
$$
p_{111}=G_{11}s_1\,,\ p_{112}=G_{12}s_1\,,\ p_{122}=G_{12}s_2\,,\ p_{222}=G_{22}s_2\,.
$$
These values for $p_{ijk}$ will define an integral element if and only if
they satisfy \eqref{eq: E^*Theta=0}, which turns out to be
\begin{equation}\label{eq: Q as s}
(G_{11}G_{22}{-}{G_{12}}^2)^2\,K 
- G_{12}(G_{11}s_2{-}G_{12}s_1)(G_{21}s_2{-}G_{22}s_1) = 0.
\end{equation}
Because $G_{12}$ and $G_{11}G_{22}-{G_{12}}^2$ are nonvanishing, this equation
defines either a hyperbola or a pair of transverse lines in the $s_1s_2$-plane,
depending on whether $K$ is zero or not.

Meanwhile, the characteristic cubic at points of $L_u$ 
is (up to a nonzero multiple) given by
$$
\begin{aligned}
C &= -\bigl((G_{11}G_{22}+{G_{12}}^2)s_1-2G_{11}G_{12}s_2\bigr)\,(\omega^1)^2\omega^2\\
& \qquad +\bigl((G_{11}G_{22}+{G_{12}}^2)s_2-2G_{22}G_{12}s_1\bigr)\,\omega^1(\omega^2)^2.
\end{aligned}
$$
Obviously $C$ will have three real distinct factors as long as both quantities 
$(G_{11}G_{22}+{G_{12}}^2)s_1-2G_{11}G_{12}s_2$ and $(G_{11}G_{22}+{G_{12}}^2)s_2-2G_{22}G_{12}s_1$ are nonzero.  The vanishing of these quantities define two lines $\ell_1$
and $\ell_2$ in the $s_1s_2$-plane that intersect transversely at $s_1=s_2=0$. 

When $K$ is nonzero, the hyperbola \eqref{eq: Q as s} in the $s_1s_2$-plane 
is not contained in the union of the two lines $\ell_1$ and $\ell_2$.  Hence
there is a choice of $(s_1,s_2)$ that satisfies \eqref{eq: Q as s} for which $C$
has three real distinct factors.

When $K=0$, one checks that the two (distinct) lines in the $s_1s_2$-plane 
defined by $G_{11}s_2{-}G_{12}s_1=0$ and $G_{21}s_2{-}G_{22}s_1=0$ 
only meet the lines $\ell_1$ and $\ell_2$ in the origin $s_1=s_2=0$.  
Hence, again there is a choice of $(s_1,s_2)$ that satisfies \eqref{eq: Q as s} 
for which $C$ has three real distinct factors.

Such a choice of $(s_1,s_2)$ defines an element of $L_u$ 
that lies in $\cV^h_2(\cI_+,\omega^1\w\omega^2)$, proving the claimed
surjectivity.
\end{proof}
 
\begin{remark}[Universal hyperbolicity]
It is not always true that each~$u\in F$ 
has a $2$-dimensional integral element~$E\subset T_uF$
for which the characteristic cubic~$E^*(C)$ is the product 
of a linear factor and an irreducible quadratic factor 
(the other open $\GL(2,\bbR)$-orbit).  For example, if $g>0$, then such 
integral elements exist only when $K(u)<0$.  In fact, when $g>0$ and $K>0$, 
one finds that \emph{all} of the K\"ahler-ordinary integral elements are hyperbolic.
\end{remark}

\begin{theorem}\label{thm: smooth}
Let~$\cI^{(1)}$ denote the contact ideal 
on~$\cV^h_2(\cI_+,\omega^1\w\omega^2)\subset \Gr_2(TF)$, endowed
with its canonical independence condition~$\omega$.  
Then~$(\cI^{(1)},\omega)$ is involutive hyperbolic in the sense
of Yang~\cite{MR0897707}.  

Consequently, any smooth nondegenerate quadratic form on~$M^2$
is locally Hessianizable.
\end{theorem}

\begin{proof}
Since~$(\cI^{(1)},\omega)$ on $\cV^h_2(\cI_+,\omega^1\w\omega^2)$
is an open subset of the ordinary prolongation 
of the involutive system~$(\cI_+,\omega^1\w\omega^2)$, 
it is involutive, with Cartan characters $s_0=4$, $s_1=3$, $s_2=0$.
Moreover, because ordinary prolongation preserves the characteristic
variety, it follows that the characteristic variety of each
integral element of~$(\cI^{(1)},\omega)$ consists of three real,
distinct points.  Thus, $(\cI^{(1)},\omega)$ is involutive hyperbolic
in the sense of Yang~\cite{MR0897707}.  

In particular, it follows from Theorem~4.11 in~\cite{MR0897707} 
that, when $g$ is a smooth nondegenerate quadratic form on~$M^2$, 
each $u\in F$ lies in a smooth, ordinary, \emph{hyperbolic} 
integral surface of~$\bigl(\cI_+,\omega^1\w\omega^2\bigr)$.
Consequently, any smooth nondegenerate quadratic form 
on a surface can locally be written in Hessian form.
\end{proof}

\bibliographystyle{hamsplain}

\providecommand{\bysame}{\leavevmode\hbox to3em{\hrulefill}\thinspace}

\end{document}